\documentclass[a4paper,12pt]{article}
\let\Horig\H
\usepackage[english]{babel}
\usepackage{tikz}
\usetikzlibrary{matrix,arrows,decorations.markings}
\usepackage{amsmath,amsfonts,amssymb,amsthm,url,textcomp}
\usepackage{csquotes}
\usepackage{a4wide}
\usepackage[numbers]{natbib}
\usepackage{fancyhdr}
\usepackage{anyfontsize}
\usepackage{hyperref}
\usepackage{hhline}
\usepackage{leftidx}

\allowdisplaybreaks

\newtheorem{theorem}{Theorem}\numberwithin{theorem}{section}
\newtheorem{definition}[theorem]{Definition}
\newtheorem{lemma}[theorem]{Lemma}
\newtheorem{corollary}[theorem]{Corollary}
\newtheorem{proposition}[theorem]{Proposition}
\newtheorem{notation}[theorem]{Notation}

\newtheorem{theoremm}{Theorem}\numberwithin{theoremm}{subsection}

\newtheorem{lemmma}[theoremm]{Lemma}
\newtheorem{corrollary}[theoremm]{Corollary}

\numberwithin{theoremmm}{subsubsection}

\theoremstyle{remark}

\newcommand{\F}{\operatorname{F}}

\newcommand{\IN}{\mathbb{N}}

\newcommand{\IZ}{\mathbb{Z}}

\newcommand{\wt}{\operatorname{wt}}
\newcommand{\Pow}{\mathcal{P}}
\renewcommand{\O}{\mathcal{O}}

\begin{document}

\title{Nilpotency and the number of word maps of a finite group}

\author{Alexander Bors\thanks{University of Salzburg, Mathematics Department, Hellbrunner Stra{\ss}e 34, 5020 Salzburg, Austria. \newline E-mail: \href{mailto:alexander.bors@sbg.ac.at}{alexander.bors@sbg.ac.at} \newline The author is supported by the Austrian Science Fund (FWF):
Project F5504-N26, which is a part of the Special Research Program \enquote{Quasi-Monte Carlo Methods: Theory and Applications}. \newline 2010 \emph{Mathematics Subject Classification}: Primary: 20D15, 20D60. Secondary: 20E05, 20E10, 20F18. \newline \emph{Key words and phrases:} Word map, Finite group, Nilpotent group.}}

\date{\today}

\maketitle

\abstract{For a finite group $G$ and a non-negative integer $d$, denote by $\Omega_d(G)$ the number of functions $G^d\rightarrow G$ that are induced by substitution into a word with variables among $X_1,\ldots,X_d$. In this note, we show that nilpotency of $G$ can be characterized through the asymptotic growth rate of $\Omega_d(G)$ as $d\to\infty$.}

\section{Introduction}\label{sec1}

\subsection{Motivation and main results}\label{subsec1P1}

In recent years, various authors have made contributions to the theory of word maps on groups; we refer interested readers to the survey article \cite{Sha13a} for an overview of results and open problems of this theory.

Recall that a \emph{word} is just an element $w$ of some free group $\F(X_1,\ldots,X_d)$ and that each such word gives rise to a \emph{word map} $w_G:G^d\rightarrow G$ on every group $G$, induced by substitution. For a fixed finite group $G$, denote by $\Omega_d(G)$ the number of functions $G^d\rightarrow G$ that are of the form $w_G$ for some $w\in \F(X_1,\ldots,X_d)$. Moreover, set $\omega_d:=\log_2{\Omega_d(G)}$ (binary logarithm). The aim of this note is to show the following:

\begin{theoremm}\label{mainTheo}
Let $G$ be a finite group, $c$ a non-negative integer.

\begin{enumerate}
\item If $G$ is nilpotent of class exactly $c$, then $\omega_d(G)=\Theta(d^c)$ as $d\to\infty$.
\item If $G$ is not nilpotent, then $\omega_d(G)\geq 2^d-1$.
\end{enumerate}
\end{theoremm}

In particular, if $G$ is nilpotent of class $c$, then for some large enough constant $C=C(G)>0$, we have $\Omega_d(G)\leq C^{d^c}$ for all non-negative integers $d$, whereas $\Omega_d(G)$ grows doubly exponentially in $d$ if $G$ is not nilpotent. Note that the trivial upper bound $\Omega_d(G)\leq |G|^{|G|^d}$ is also doubly exponential in $d$ for fixed $G$. By Theorem \ref{mainTheo}, we also have the following quantitative characterization of nilpotency of finite groups:

\begin{corrollary}\label{mainCor}
A finite group $G$ is nilpotent (of class exactly $c$) if and only if the sequence $(\omega_d(G))_{d\geq 0}$ is of polynomial growth (of degree exactly $c$).
\end{corrollary}

\subsection{Notation}\label{subsec1P2}

We denote by $\IN$ the set of natural numbers (including $0$) and by $\IN^+$ the set of positive integers. The set of all \emph{nonempty} subsets of a set $X$ is denoted by $\Pow_{{}\not=\emptyset}(X)$. The exponent of a finite group $G$ is denoted by $\exp(G)$. As in the previous subsection, we denote by $\F(X_1,\ldots,X_d)$ the free group on the formal generators $X_1,\ldots,X_d$, and we will denote by $\F_c(X_1,\ldots,X_d)$ the free nilpotent group of class exactly $c$ on the formal generators $X_1,\ldots,X_d$.

\subsection{Normal forms of elements in free nilpotent groups}\label{subsec1P3}

For the reader's convenience, we briefly recall a result on normal forms of elements in free nilpotent groups (based on a certain polycyclic decomposition of such groups) which we will need, see the blog post \cite{Tao09a} or the textbook \cite[Theorem 5.13A, p.~343, and its proof]{MKS66a} for more details.

Fix $d,c\in\IN$, and assign to each $i\in\{1,\ldots,d\}$ the \emph{weight} $\wt(i):=1$. Consider formal iterated commutators of the numbers $1,\ldots,d$ (we will henceforth call them \emph{formal $d$-commutators}) and define the \emph{weight} of a formal $d$-commutator recursively via $\wt([\alpha,\beta]):=\wt(\alpha)+\wt(\beta)$. Moreover, for a formal $d$-commutator $\gamma$, define $X_{\gamma}$, a word in the variables $X_1,\ldots,X_d$, recursively via $X_{[\alpha,\beta]}:=[X_{\alpha},X_{\beta}]=X_{\alpha}^{-1}X_{\beta}^{-1}X_{\alpha}X_{\beta}$.

Then there exists an explicitly constructible finite ordered tuple $(\alpha_1^{(d,c)},\ldots,\alpha_{N_{d,c}}^{(d,c)})$ of pairwise distinct formal $d$-commutators each of weight at most $c$ such that every element of $\F_c(X_1,\ldots,X_d)$ has a unique representation of the form $\prod_{j=1}^{N_{d,c}}{(X_{\alpha_j^{(d,c)}})^{k_j}}$ with $k_j\in\IZ$.

For later use, we also note the following, which can be easily proved by induction on $c$:

\begin{lemmma}\label{formalLem}
For each $c\in\IN$, there exists a polynomial $P_c(X)\in\IZ[X]$ of degree $c$ such that for every $d\in\IN$, the number of formal $d$-commutators of weight at most $c$ is exactly $P_c(d)$.\qed
\end{lemmma}

See also \cite[Problems for Section 5.2, Problem 5]{MKS66a} for a precise formula for the coefficients of $P_c(d)$.

\section{Proof of Theorem \ref{mainTheo}}\label{sec2}

We begin by providing a lower bound on $\Omega_d(G)$ which will yield both statement (2) of Theorem \ref{mainTheo} and half of statement (1). For this, we need some notation.

\begin{notation}\label{expNot}
Let $G$ be a group.

\begin{enumerate}
\item For $r\in\IN^+$ and elements $g_1,\ldots,g_r\in G$, we define their \emph{nested commutator} $[g_1,\ldots,g_r]$ recursively via $[g_1]:=g_1$ and $[g_1,\ldots,g_{r+1}]:=[[g_1,\ldots,g_r],g_{r+1}]$.
\item If $G$ is finite and $r\in\IN^+$, then denote by $\exp_r(G)$ the least common multiple of the orders of the elements of $G$ of the form $[g_1,\ldots,g_r]$, where the $g_i$ range over $G$.
\end{enumerate}
\end{notation}

Note that for fixed finite $G$, the sequence $(\exp_r(G))_{r\geq 1}$ is monotonically decreasing and that $\exp_1(G)=\exp(G)$.

\begin{definition}\label{niceWordDef}
Let $G$ be a finite group, $d\in\IN$.

\begin{enumerate}
\item A function $f:\Pow_{{}\not=\emptyset}(\{1,\ldots,d\})\rightarrow\IN$ is called \emph{$G$-admissible} if and only if for each $r\in\{1,\ldots,d\}$, we have $f(M)<\exp_r(G)$ for every $r$-element subset $M\subseteq\{1,\ldots,d\}$.
\item To each $G$-admissible function $f$, we assign a word $w_f(X_1,\ldots,X_d)$, defined as follows:

\[
w_f(X_1,\ldots,X_d):=\prod_{r=1}^d\prod_{i_1=1}^{d}\prod_{i_2=i_1+1}^{d}\cdots\prod_{i_r=i_{r-1}+1}^{d}{[X_{i_1},\ldots,X_{i_r}]^{f(\{i_1,\ldots,i_r\})}}.
\]
\end{enumerate}
\end{definition}

We now show:

\begin{proposition}\label{niceWordProp}
Let $G$ be a finite group, $d\in\IN$.

\begin{enumerate}
\item Let $f,g$ be distinct $G$-admissible functions $\Pow_{{}\not=\emptyset}(\{1,\ldots,d\})\rightarrow\IN$. Then the word maps $(w_f)_G$ and $(w_g)_G$ are distinct functions $G^d\rightarrow G$.
\item $\Omega_d(G)\geq\prod_{r=1}^d{\exp_r(G)^{{d \choose r}}}$.
\end{enumerate}
\end{proposition}

\begin{proof}
Statement (2) follows from statement (1), as the asserted lower bound is by definition just the number of $G$-admissible functions $\Pow_{{}\not=\emptyset}(\{1,\ldots,d\})\rightarrow\IN$.

As for statement (1), let $r\in\IN^+$ be minimal such that $f$ and $g$ disagree on some $r$-element subset of $\{1,\ldots,d\}$ and let $i_1<i_2<\cdots<i_r$ be such that $\{i_1,\ldots,i_r\}$ is minimal with respect to the lexicographical ordering among all $r$-element subsets of $\{1,\ldots,d\}$ on which $f$ and $g$ have different values. Set $a:=f(\{i_1,\ldots,i_r\})-g(\{i_1,\ldots,i_r\})$ and note that $a\in\{1,\ldots,\exp_r(G)-1\}$. By definition of $\exp_r(G)$, this means that we can fix $g_{i_1},\ldots,g_{i_r}\in G$ such that

\[
[g_{i_1},\ldots,g_{i_r}]^{f(\{i_1,\ldots,i_r\})-g(\{i_1,\ldots,i_r\})}\not=1_G.
\]

Moreover, set $g_i:=1_G$ for $i\in\{1,\ldots,d\}\setminus\{i_1,\ldots,i_r\}$.

We claim that $(w_f)_G$ and $(w_g)_G$ disagree on the argument $(g_1,\ldots,g_d)$. To see this, note that by choice of $(i_1,\ldots,i_r)$, we have

\[
w_f=w_0\cdot[X_{i_1},\ldots,X_{i_r}]^{f(\{i_1,\ldots,i_r\})}\cdot w_1
\]

and

\[
w_g=w_0\cdot[X_{i_1},\ldots,X_{i_r}]^{g(\{i_1,\ldots,i_r\})}\cdot w_2,
\]

where $w_0,w_1,w_2\in\F(X_1,\ldots,X_d)$ and $w_1$, $w_2$ are products of nested commutators of length at least $r$ distinct from $[X_{i_1},\ldots,X_{i_r}]$, so that when substituting $g_i$ for $X_i$, at least one of the entries of each such nested commutator will be $1_G$, forcing it to be trivial. Therefore,

\begin{align*}
&(w_f)_G(g_1,\ldots,g_d)=w_0(g_1,\ldots,g_d)\cdot[g_{i_1},\ldots,g_{i_r}]^{f(\{i_1,\ldots,i_r\})} \\
&\not=w_0(g_1,\ldots,g_d)[g_{i_1},\ldots,g_{i_r}]^{g(\{i_1,\ldots,i_r\})}=(w_g)_G(g_1,\ldots,g_d),
\end{align*}

as required.
\end{proof}

\begin{corollary}\label{niceWordCor}
Let $G$ be a finite group, $c\in\IN$, and assume that the $(c+1)$-th term in the lower central series of $G$ is nontrivial. Then for every $d\in\IN$, $\omega_d(G)\geq\sum_{r=1}^c{{d \choose r}}$.
\end{corollary}

\begin{proof}
By assumption, $\exp_r(G)\geq 2$ for $r=1,\ldots,c$, so the result follows immediately from Proposition \ref{niceWordProp}(2).
\end{proof}

In particular, $d^c=\O(\omega_d(G))$ whenever $G$ is nilpotent of class exactly $c$ (which is \enquote{half of Theorem \ref{mainTheo}(1)}) and if $G$ is not nilpotent, we can choose $c:=d$ in Corollary \ref{niceWordCor} and get Theorem \ref{mainTheo}(2) using that $\sum_{r=1}^d{{d \choose r}}=2^d-1$. The second half of Theorem \ref{mainTheo}(1) is provided by the following:

\begin{lemma}\label{upperBoundLem}
Let $c\in\IN$, and let $P_c(X)\in\IZ[X]$ be as in Lemma \ref{formalLem}. Then for every finite nilpotent group $G$ of class exactly $c$, we have $\omega_d(G)\leq P_c(c)\cdot\log_2(\exp(G))$.
\end{lemma}

\begin{proof}
We show the equivalent $\Omega_d(G)\leq\exp(G)^{P_c(c)}$. Let $w\in \F(X_1,\ldots,X_d)$. Then the remarks in Subsection \ref{subsec1P3} yield that

\[
w_{\F_c(X_1,\ldots,X_d)}(X_1,\ldots,X_d)=\prod_{j=1}^{N_{d,c}}((X_{\alpha_j^{(d,c)}})^{k_j}).
\]

for suitable integers $k_j$. As $\F_c(X_1,\ldots,X_d)$ is the free object in the category of nilpotent groups of class at most $c$, this relation among its generators translates to an identity in all nilpotent groups of class at most $c$, yielding in particular that $w_G=(\prod_{j=1}^{N_{d,c}}((X_{\alpha_j^{(d,c)}})^{k_j}))_G$. Hence each word map on $G$ is induced by a word of the form $\prod_{j=1}^{N_{d,c}}((X_{\alpha_j^{(d,c)}})^{k_j})$ for suitable integers $k_j$, which we may w.l.o.g.~assume to be from the finite range $\{0,\ldots,\exp(G)-1\}$. Therefore, $\Omega_d(G)\leq\exp(G)^{N_{d,c}}\leq\exp(G)^{P_c(c)}$, as required.
\end{proof}

\end{document}